\documentclass[12pt]{article}
\usepackage{amsmath,amssymb,fullpage}
\usepackage[applemac]{inputenc}

\newenvironment{proof}[1][Proof]{\noindent\textbf{#1.} }{\ \rule{0.5em}{0.5em}}

\newtheorem{definition}{Definition}[section]

\newtheorem{theorem}[definition]{Theorem}

\numberwithin{equation}{section}

\def\1B{\text{1\!\!I}}

\begin{document}

\title{Infinite horizon optimal control of forward- backward stochastic differential
equations with delay.}
\author{\ \ Nacira AGRAM\thanks{Laboratory of Applied Mathematics, University Med
Khider, Po. Box 145, Biskra $\left(  07000\right)  $ Algeria. Email:
agramnacira@yahoo.fr} \ and Bernt \O KSENDAL\thanks{Center of Mathematics for
Applications (CMA), University of Oslo, Box 1053 Blindern, N-0316 Oslo,
Norway. Email: oksendal@math.uio.no} \thanks{The research leading to these
results has received funding from the European Research Council under the
European Community's Seventh Framework Programme (FP7/2007-2013) / ERC grant
agreement no [228087].}}
\date{\ \ 15 May 2013}
\maketitle

\begin{abstract}
We consider a problem of optimal control of an infinite horizon system
governed by forward-backward stochastic differential equations with delay.
Sufficient and necessary maximum principles for optimal control under partial
information in infinite horizon are derived. We illustrate our results by an
application to a problem of optimal consumption with respect to recursive
utility from a cash flow with delay.

\end{abstract}


\textsf{[2010]MSC:   93EXX; 93E20; 60J75; 60H10; 60H20; 34K50}

\section{Introduction}

One of the problems posed recently and which has got a lot of attention\ is
the optimal control of forward-backward stochastic differential equations
(FBSDEs). This theory was first developed in the early $90$s by \cite{A},
\cite{MPY}, \cite{Pe} and others.

The paper \cite{Pe} established the maximum principle of FBSDE in the convex
setting and later it was studied by many authors such as \cite{AG}, \cite{MO},
\cite{OSM}, \cite{OSF}, \cite{X}. For the existence of an optimal control of
FBSDEs, see \cite{BGM}.

The optimal control problem of FBSDE has interesting applications especially
in finance like in option pricing and recursive utility problems. The latter
was introduced by \cite{DE} and for more details about the recursive utility
maximization problems, we refer to \cite{BS}, \cite{Pe}.

The recursive utility is a solution of the backward stochastic differential
equation (BSDE) which is not necessarily linear. The BSDE was studied by
\cite{P}, \cite{PP} etc.

All the papers above where dealing with finite horizon FBSDEs. Other related
stochastic control publications dealing with finite horizon only are
\cite{BL}, \cite{ME} and \cite{TL}.

Related papers dealing with infinite horizon control, but either without FB
systems or without delay, are \cite{AHOP}, \cite{HOP}, \cite{MV}, \cite{PS}
and \cite{Y}.

We will study this problem by using a version of the maximum principle which
is a combination of the infinite horizon maximum principle in \cite{AHOP} and
the finite horizon maximum principle for FBSDEs in \cite{OSM} and \cite{MO}.
We extend an application in \cite{MO} to infinite horizon and in \cite{OSZ}
for the FBSDE.

We emphasize that although the current paper has similarities with
\cite{AHOP}, the fact that we are considering forward-backward systems and not
just forward systems creates a new situation. In particular, we now get
additional transversality conditions involving the additional adjoint process
$\lambda$. See Theorem $2.1$ and Theorem $3.1$.

In this paper we obtain a sufficient and a necessary maximum principle for
infinite horizon control of \ FBSDEs with delay. As an illustration we solve
explicitly an infinite horizon optimal consumption problem with recursive utility.

The partial results mentioned above indicate that it should be possible to
prove a general existence and uniqueness theorem for controlled infinite
horizon FBSDEs with delay. However, this is difficult problem and we leave
this for future research.

\section{Setting of the problem}

Let $\left(  \Omega,\mathcal{F},\mathbb{F}=(\mathcal{F}_{t})_{t\geq
0},P\right)  $ be a complete filtered probability space on which a
one-dimensional standard Brownian motion $B\left(  t\right)  $ and an
independent compensated Poisson random measure $\tilde{N}(dt,da)=N(dt,da)-\nu
(da)dt$ are defined. We assume that $\mathbb{F}$ is the natural filtration,
made right continuous generated by the processes $B$ and $N$.

We study the following infinite horizon coupled forward-backward stochastic
differential equations control system with delay:

\bigskip

FORWARD EQUATION in the unknown measurable process $X^{u}(t)$:
\begin{equation}%
\begin{array}
[c]{l}%
dX(t)=dX^{u}(t)=b\left(  t,\boldsymbol{X}^{u}(t),u(t)\right)  dt+\sigma\left(
t,\boldsymbol{X}^{u}(t),u(t)\right)  dB(t)\\
\text{ \ \ \ \ \ \ \ \ }+%
{\textstyle\int\limits_{\mathbb{R}_{0}}}
\theta\left(  t,\boldsymbol{X}^{u}(t),u(t),a\right)  \tilde{N}(dt,da);t\in
\left[  0,\infty\right)  ,\\
X(t)=X_{0}(t);\text{ \ \ \ \ \ \ \ \ \ \ \ \ \ \ \ \ \ \ \ }t\in\left[
-\delta,0\right]  ,
\end{array}
\label{eq1.1}%
\end{equation}

where%

\[
\boldsymbol{X}^{u}\mathbb{(}t\mathbb{)=}\left(  X^{u}(t),X_{1}^{u}%
(t),X_{2}^{u}(t)\right)  \text{,}%
\]

with \ \ \ \
\[
X_{1}^{u}(t)=X^{u}(t-\delta),\text{ }X_{2}^{u}(t)=%
{\textstyle\int\limits_{t-\delta}^{t}}
e^{-\rho(t-r)}X^{u}(r)dr,
\]

and $X_{0}$ is a given continuous (and deterministic) function on $\left[
-\delta,0\right]  $. \newline

BACKWARD EQUATION in the unknown measurable processes $Y^{u}(t),Z^{u}%
(t),K^{u}\left(  t,\cdot\right)  $:
\begin{equation}%
\begin{array}
[c]{l}%
dY^{u}(t)=-g\left(  t,\boldsymbol{X}^{u}(t),Y^{u}(t),Z^{u}(t),u(t)\right)
dt+Z^{u}(t)dB(t)\\
\text{ \ \ \ \ \ \ \ \ }+%
{\textstyle\int\limits_{\mathbb{R}_{0}}}
K^{u}\left(  t,a\right)  \tilde{N}(dt,da);t\in\left[  0,\infty\right)  .
\end{array}
\label{eq1.2}%
\end{equation}

We interpret the infinite horizon BSDE $\left(  \ref{eq1.2}\right)  $ in the
sense of Pardoux \cite{P} i.e. for all $T<\infty$, the triple $(Y^{u}%
(t),Z^{u}(t),K^{u}\left(  t,\cdot\right)  )$ solves the equation%

\begin{equation}%
\begin{array}
[c]{l}%
Y^{u}(t)=Y(T)+%
{\textstyle\int\limits_{t}^{T}}
g\left(  s,\boldsymbol{X}^{u}(s),Y^{u}(s),Z^{u}(s),u(s)\right)  ds-%
{\textstyle\int\limits_{t}^{T}}
Z^{u}(s)dB(s)\\
\text{ \ \ \ \ \ \ \ \ }-%
{\textstyle\int\limits_{t}^{T}}
{\textstyle\int\limits_{\mathbb{R}_{0}}}
K^{u}\left(  s,a\right)  \tilde{N}(ds,da);\text{ }0\leq t\leq T\text{.}%
\end{array}
\label{eq1.3}%
\end{equation}

We call the process $(Y^{u}(t),Z^{u}(t),K^{u}\left(  t,\cdot\right)  )$ the
solution of $\left(  \ref{eq1.3}\right)  $ if it also satisfies
\begin{equation}
E\underset{t\geq0}{[\sup\text{ }}e^{\kappa t}\left(  Y^{u}\right)  ^{2}(t)+%
{\textstyle\int\limits_{0}^{\infty}}
e^{\kappa t}(\left(  Z^{u}\right)  ^{2}(t)+%
{\textstyle\int\limits_{\mathbb{R}_{0}}}
\left(  K^{u}\right)  ^{2}\left(  t,a\right)  \nu(da))dt]<\infty\label{eq1.4}%
\end{equation}
for all constants $\kappa>0$. We refer the reader to Section $4$ in \cite{P}
for assumptions of the coefficients that insure the existence and uniqueness
of the solution of the FBSDE system. \newline Note that $\left(
\ref{eq1.4}\right)  $ implies in particular that $%
\begin{array}
[c]{c}%
\underset{t\rightarrow\infty}{\text{lim}}Y^{u}(t)=0
\end{array}
$.

Throughout this paper, we introduce the following notations%

\[%
\begin{array}
[c]{l}%
\delta>0,\rho>0\text{ are given constants,}\\
b:[0,\infty)\times\mathbb{R}^{3}\times\mathcal{U}\times\Omega\rightarrow
\mathbb{R},\\
\sigma:[0,\infty)\times\mathbb{R}^{3}\times\mathcal{U}\times\Omega
\rightarrow\mathbb{R},\\
g:[0,\infty)\times\mathbb{R}^{5}\times\mathcal{U}\times\Omega\rightarrow
\mathbb{R},\\%
\mathbb{R}
_{0}:=%
\mathbb{R}
-\{0\},\\
\theta,K:[0,\infty)\times\mathbb{R}^{3}\times\mathcal{U}\times\mathbb{R}%
_{0}\times\Omega\rightarrow\mathbb{R},\\
f:[0,\infty)\times\mathbb{R}^{5}\times\mathcal{R}\times\mathcal{U}\times
\Omega\rightarrow\mathbb{R},\\
h:\mathbb{R\rightarrow R},
\end{array}
\]

where the coefficients $b,\sigma,\theta$ and $g$ are FrÈchet differentiable
$(C^{1})$ with respect to the variables $\left(  \boldsymbol{x},y,z,u\right)
$. Here $\mathcal{R}$ is the set of all functions $%
\begin{array}
[c]{c}%
k:%
\mathbb{R}
_{0}\rightarrow%
\mathbb{R}
.
\end{array}
$In the following, we will for simplicity suppress the dependence on
$\omega\in\Omega$ in the notation.

Note that if $g$ does not depend on $Y^{u}(s)$ and $Z^{u}(s)$ then the ItÙ
representation theorem for LÈvy processes ( see \cite{GOP}), implies that
equation $\left(  \ref{eq1.3}\right)  $ is equivalent to the equation%

\begin{equation}
Y^{u}(t)=E[Y(T)+%
{\textstyle\int\limits_{t}^{T}}
g\left(  s,\boldsymbol{X}^{u}(s),u(s)\right)  ds\mid\mathcal{F}_{t}]\text{;
}t\leq T\text{, for all }T<\infty\text{.} \label{eq1.5}%
\end{equation}

Let $\mathbb{E}=\{\mathcal{E}_{t}\}_{t\geq0}$ with $\mathcal{E}_{t}%
\subseteq\mathcal{F}_{t}$ for all $t\geq0$ be a given subfiltration,
representing the information available to the controller at time $t$.

Let $\mathcal{U}$ be a non-empty convex subset of $\mathbb{R}.$ We let
$\mathcal{A=A}_{\mathcal{E}}$ denote a given locally convex family of
admissible $\mathbb{E}$-predictable control processes $u$ with values in
$\mathcal{U}$, such that the corresponding solution $(X^{u}$, $Y^{u}$, $Z^{u}%
$, $K^{u})$ of $\left(  \ref{eq1.1}\right)  -\left(  \ref{eq1.5}\right)  $
exist and
\[
E[\int\limits_{0}^{\infty}\left\vert X^{u}(t)\right\vert ^{2}dt]<\infty.
\]

The corresponding performance functional is%

\begin{equation}
J(u)=E[%
{\textstyle\int\limits_{0}^{\infty}}
f\left(  t,\boldsymbol{X}(t)\right)  \text{ }dt+h(Y(0))], \label{eq1.6}%
\end{equation}

where $f\left(  t,\boldsymbol{X}(t)\right)  $ is a short-hand notation for
$f\left(  t,\boldsymbol{X}^{u}(t),Y^{u}(t),Z^{u}(t),K^{u}\left(
t,\cdot\right)  ,u(t)\right)  $.

We assume that the functions $f$ and $h$ are FrÈchet differentiable $(C^{1})$
with respect to the variables $\left(  \boldsymbol{x},y,z,k(\cdot),u\right)  $
and $Y(0)$, respectively, and $f$ satisfies%

\begin{equation}
E[%
{\textstyle\int\limits_{0}^{\infty}}
\left\vert f\left(  t,\boldsymbol{X}(t)\right)  \text{ }\right\vert
dt]<\infty\text{, for all }u\in\mathcal{A}\text{.} \label{eq1.7}%
\end{equation}

The optimal control problem is to find an optimal control $u^{\ast}%
\in\mathcal{A}$ and the value function $\Phi:$ $C(\left[  -\delta,0\right]
)\rightarrow%
\mathbb{R}
$ such that%

\begin{equation}
\Phi(X_{0})=\underset{u\in\mathcal{A}}{\text{sup}}J(u)=J(u^{\ast}).
\label{eq1.8}%
\end{equation}

We will study this problem by using a version of the maximum principle which
is a combination of the infinite horizon maximum principle in \cite{AHOP} and
the finite horizon maximum principle for FBSDEs in \cite{OSF} and \cite{OSM}.

The Hamiltonian%

\[
H:[0,\infty)\times%
\mathbb{R}
^{5}\times L^{2}(\nu)\times\mathcal{U}\times%
\mathbb{R}
^{3}\times L^{2}(\nu)\rightarrow%
\mathbb{R}
\]

is defined by%

\begin{equation}%
\begin{array}
[c]{c}%
H(t,\boldsymbol{x},y,z,k(\cdot),u,\lambda,p,q,r(\cdot))=f(t,\boldsymbol{x}%
,y,z,k,u)+g(t,\boldsymbol{x},y,z,u)\lambda\\
\text{ \ \ \ \ \ \ \ \ }+b(t,\boldsymbol{x},u)p+\sigma(t,\boldsymbol{x},u)q+%
{\textstyle\int\limits_{\mathbb{R}_{0}}}
\theta(t,\boldsymbol{x},u,a)r(a)\nu(da)\text{.}%
\end{array}
\label{eq1.9}%
\end{equation}

We assume that the Hamiltonian $H$ is FrÈchet differentiable $(C^{1})$ in the
variables $\boldsymbol{x},y,z,k$ and $u$.

We also assume that for all $t$ the FrÈchet derivative of $H(t,\boldsymbol{X}%
^{u}(t),Y^{u}(t),Z^{u}(t),k,u(t),p(t),q(t),r(t,\cdot))$ with respect to $k$,
denoted by $\nabla_{k}H(t,\cdot)$, as a random measure is absolutely
continuous with respect to $\nu$, with Radon-Nikodym derivative $\frac
{d\nabla_{k}H}{d\nu}$ satisfying%
\[
E[\int\limits_{0}^{T}\int\limits_{%
\mathbb{R}
}\left\vert \frac{d\nabla_{k}H}{d\nu}(t,a)\right\vert ^{2}\nu(da)dt]<\infty
\text{, for all }T<\infty.
\]

See Appendix A in \cite{OSF} for details.

We associate to the problem $\left(  \ref{eq1.8}\right)  $ the following pair
of forward-backward SDEs in the adjoint processes $\lambda(t)$, $($
$p(t),q(t),r(t,\cdot))$:

\bigskip

\bigskip

ADJOINT FORWARD EQUATION:%

\begin{equation}
\left\{
\begin{array}
[c]{l}%
d\lambda(t)=\dfrac{\partial H}{\partial y}(t)\text{ }dt+\dfrac{\partial
H}{\partial z}(t)\text{ }dB(t)+%
{\displaystyle\int\limits_{\mathbb{R}_{0}}}
\frac{d\nabla_{k}H}{d\nu}(t,a)\tilde{N}(dt,da)\\
\lambda(0)=h^{^{\prime}}(Y(0))
\end{array}
\right.  \label{eq1.10}%
\end{equation}

where we have used the short hand notation%
\[
\tfrac{\partial H}{\partial y}(t)=\tfrac{\partial}{\partial y}%
H(t,\boldsymbol{X}^{u}(t),y,Z^{u}(t),K^{u}(t,\cdot),u(t),\lambda
(t),p(t),q(t),r(t,\cdot))\mid_{y=Y(t)}%
\]

and similarly with $\tfrac{\partial H}{\partial z}(t),\tfrac{\partial
H}{\partial x}(t),...$

\bigskip

ADJOINT BACKWARD EQUATION:%

\begin{equation}
dp(t)=E[\mu(t)\mid\mathcal{F}_{t}]dt+q(t)dB(t)+%
{\textstyle\int\limits_{\mathbb{R}_{0}}}
r(t,a)\tilde{N}(dt,da);t\in\left[  0,\infty\right)  \label{eq1.11}%
\end{equation}

where%

\begin{equation}
\mu(t)=-\tfrac{\partial H}{\partial x}(t)-\tfrac{\partial H}{\partial x_{1}%
}(t+\delta)-e^{\rho t}(%
{\textstyle\int\limits_{t}^{t+\delta}}
\tfrac{\partial H}{\partial x_{2}}(s)e^{-\rho s}ds)\text{.} \label{eq1.12}%
\end{equation}

with terminal condition as in $\left(  \ref{eq1.4}\right)  $, i.e.%

\begin{equation}
E\underset{t\geq0}{[\sup\text{ }}e^{\kappa t}p^{2}(t)+%
{\textstyle\int\limits_{0}^{\infty}}
e^{\kappa s}(q^{2}(s)+%
{\textstyle\int\limits_{\mathbb{R}_{0}}}
r^{2}\left(  s,a\right)  \nu(da))ds]<\infty,\nonumber
\end{equation}

for all constants $\kappa>0$.

The unknown process $\lambda(t)$ is the adjoint process corresponding to the
backward system $(Y(t),Z(t),K(t,\cdot))$ and the triple unknown
$(p(t),q(t),r(t,\cdot))$ is the adjoint process corresponding to the forward
system $X(t)$.

We show that in this infinite horizon setting the appropriate terminal
conditions for the BSDEs for $(Y(t),Z(t),K(t,\cdot))$ and
$(p(t),q(t),r(t,\cdot))$ should be replaced by asymptotic transversality
conditions. See $(H_{3})$ and $(H_{6})$ below.

\section{Sufficient maximum principle for partial information}

We will prove in this section that under some assumptions the maximization of
the Hamiltonian leads to an optimal control.

\begin{theorem}
Let $\hat{u}\in\mathcal{A}$ with corresponding solutions $\boldsymbol{\hat{X}%
}(t),\hat{Y}(t),\hat{Z}(t),\hat{K}(t,\cdot),\hat{p}(t),\hat{q}(t),\hat
{r}(t,\cdot)$ and $\hat{\lambda}(t)$ of equations $\left(  \ref{eq1.1}\right)
$, $\left(  \ref{eq1.2}\right)  $, $\left(  \ref{eq1.10}\right)  $ and
$\left(  \ref{eq1.11}\right)  $. Suppose that

\begin{description}
\item[$(H_{1})$] (Concavity)
\end{description}

The functions $x\rightarrow h(x)$ and%
\[
(\boldsymbol{x},y,z,k(\cdot),u)\rightarrow H(t,\boldsymbol{x},y,z,k(\cdot
),u,\hat{\lambda}(t),\hat{p}(t),\hat{q}(t),\hat{r}(t,\cdot))
\]
are concave, for all $t\in\left[  0,\infty\right)  $.

\begin{description}
\item[$(H_{2})$] (The conditional maximum principle)%
\begin{align*}
&  \underset{v\in\mathcal{U}}{\max}E[H(t,\boldsymbol{\hat{X}}(t),\hat
{Y}(t),\hat{Z}(t),\hat{K}(t,\cdot),v,\hat{\lambda}(t),\hat{p}(t),\hat
{q}(t),\hat{r}(t,\cdot))\mid{\mathcal{E}}_{t}]\\
&  =E[H(t,\boldsymbol{\hat{X}}(t),\hat{Y}(t),\hat{Z}(t),\hat{K}(t,\cdot
),\hat{u}(t),\hat{\lambda}(t),\hat{p}(t),\hat{q}(t),\hat{r}(t,\cdot
))\mid{\mathcal{E}}_{t}]\text{.}%
\end{align*}

\end{description}

Moreover, suppose that for any $u\in\mathcal{A}$ with corresponding solutions
$\boldsymbol{X}(t),Y(t),Z(t),K(t,\cdot),p(t),$

$q(t),r(t,\cdot)$ and $\lambda(t)$ we have:

\begin{description}
\item[$(H_{3})$] (Transversality conditions)%
\[
\underset{T\rightarrow\infty}{\underline{\lim}}E[\mathbf{\ }\hat
{p}(T)\bigtriangleup\hat{X}(T)]\leq0
\]

\end{description}

and%
\[
\underset{T\rightarrow\infty}{\overline{\lim}}E\mathbf{\ [}\hat{\lambda
}(T)\bigtriangleup\hat{Y}(T)]\geq0\text{.}%
\]

where $%
\begin{array}
[c]{c}%
\bigtriangleup\hat{X}(T)=\hat{X}(T)-X(T),\text{ }\bigtriangleup\hat{Y}%
(T)=\hat{Y}(T)-Y(T).
\end{array}
$

\begin{description}
\item[$(H_{4})$] ( Growth conditions I) Suppose that for all $T<\infty$ the
following holds:%
\begin{align}
&  E[%
{\textstyle\int\limits_{0}^{T}}
\{(\bigtriangleup\hat{Y}(t))^{2}\{(\tfrac{\partial\hat{H}}{\partial y}%
(t))^{2}+%
{\textstyle\int\limits_{\mathbb{R}_{0}}}
\left\Vert \nabla_{k}\hat{H}(t,a)\right\Vert ^{2}\nu(da)\}\nonumber\\
&  +\hat{\lambda}^{2}(t)\{(\bigtriangleup\hat{Z}(t))^{2}+%
{\textstyle\int\limits_{\mathbb{R}_{0}}}
(\bigtriangleup\hat{K}(t,a))^{2}\nu(da)\}\nonumber\\
&  +(\bigtriangleup\hat{X}(t))^{2}\{\hat{q}^{2}(t)+%
{\textstyle\int\limits_{\mathbb{R}_{0}}}
\hat{r}^{2}(t,a)\nu(da)\}\nonumber\\
&  +\hat{p}^{2}(t)\{(\bigtriangleup\hat{\sigma}(t))^{2}+%
{\textstyle\int\limits_{\mathbb{R}_{0}}}
(\bigtriangleup\hat{\theta}(t,a))^{2}\nu(da)\}\}dt]<\infty. \label{eq2.4}%
\end{align}

\item[$(H_{5})$] ( Growth conditions II) Suppose that%
\begin{equation}%
\begin{array}
[c]{l}%
E[%
{\textstyle\int\limits_{0}^{T}}
\{\left\vert \hat{\lambda}(t)\bigtriangleup\hat{g}(t)\right\vert +\left\vert
\bigtriangleup\hat{Y}(t)\tfrac{\partial\hat{H}}{\partial y}(t)\right\vert
+\left\vert \bigtriangleup\hat{Z}(t)\tfrac{\partial\hat{H}}{\partial
z}(t)\right\vert \\
+%
{\textstyle\int\limits_{\mathbb{R}_{0}}}
\left\vert \nabla_{k}\hat{H}(t,a)\bigtriangleup\hat{K}(t,a)\right\vert
\nu(da)+\left\vert \bigtriangleup\hat{H}(t)\right\vert +\left\vert
\bigtriangleup\hat{b}(t)\hat{p}(t)\right\vert \\
+\left\vert \bigtriangleup\hat{\sigma}(t)\hat{q}(t)\right\vert +%
{\textstyle\int\limits_{\mathbb{R}_{0}}}
\left\vert \bigtriangleup\hat{\theta}(t,a)\hat{r}(t,a)\right\vert \nu(da)\\
+\left\vert \bigtriangleup\hat{X}(t)\tfrac{\partial\hat{H}}{\partial
x}(t)\right\vert +\left\vert \bigtriangleup\hat{u}(t)\tfrac{\partial\hat{H}%
}{\partial u}(t)\right\vert \}dt]<\infty.
\end{array}
\label{2.5}%
\end{equation}

\end{description}

where$%
\begin{array}
[c]{c}%
\sigma(t)=\sigma(t,\boldsymbol{X}(t),u(t))
\end{array}
,$ $%
\begin{array}
[c]{c}%
\hat{\sigma}(t)=\sigma(t,\boldsymbol{\hat{X}}(t),\hat{u}(t))
\end{array}
$ etc.

Then $\hat{u}$ is an optimal control for $\left(  \ref{eq1.8}\right)  $, i.e.%
\[
J(\hat{u})=\underset{u\in\mathcal{A}}{\text{sup}}J(u)\text{.}%
\]

\end{theorem}

\begin{proof}
Proof. Assume that $u\in\mathcal{A}$. We want to prove that $%
\begin{array}
[c]{c}%
J(\hat{u})-J(u)\geq0\text{, i.e. }\hat{u}\text{ is an optimal control.}%
\end{array}
$

We put
\begin{equation}
J(\hat{u})-J(u)=I_{1}+I_{2}, \label{eq2.5}%
\end{equation}

where%
\[
I_{1}=E[%
{\textstyle\int\limits_{0}^{\infty}}
\{\hat{f}(t)-f(t)\}\text{ }dt],
\]

and%
\[
I_{2}=E[h(\hat{Y}(0))-h(Y(0))].
\]

By the definition of $H$, we have%
\begin{equation}
I_{1}=E[%
{\textstyle\int\limits_{0}^{\infty}}
\{\bigtriangleup\hat{H}(t)-\bigtriangleup\hat{g}(t)\hat{\lambda}%
(t)-\bigtriangleup\hat{b}(t)\hat{p}(t)-\bigtriangleup\hat{\sigma}(t)\hat
{q}(t)-%
{\textstyle\int\limits_{\mathbb{R}_{0}}}
\bigtriangleup\hat{\theta}(t,a)\hat{r}(t,a)\nu(da)\}dt], \label{eq2.6}%
\end{equation}

where we have used the simplified notation%
\[%
\begin{array}
[c]{l}%
\hat{H}(t)=H(t,\boldsymbol{\hat{X}}(t),\hat{Y}(t),\hat{Z}(t),\hat{K}%
(t,\cdot),\hat{u}(t),\hat{\lambda}(t),\hat{p}(t),\hat{q}(t),\hat{r}%
(t,\cdot))\\
H(t)=H(t,\boldsymbol{X}(t),Y(t),Z(t),K(t,\cdot),u(t),\hat{\lambda}(t),\hat
{p}(t),\hat{q}(t),\hat{r}(t,\cdot))\text{ etc.}%
\end{array}
\]

Since $h$ is concave, we have%
\[
h(\hat{Y}(0))-h(Y(0))\geq h^{^{\prime}}(\hat{Y}(0))\bigtriangleup\hat
{Y}(0)=\hat{\lambda}(0)\bigtriangleup\hat{Y}(0).
\]

By ItÙ's formula, $(H_{4})$, $\left(  \ref{eq1.2}\right)  $ and $\left(
\ref{eq1.10}\right)  $, we have for all $T$
\begin{equation}%
\begin{array}
[c]{l}%
E[\hat{\lambda}(0)\bigtriangleup\hat{Y}(0)]=E[\hat{\lambda}(T)\bigtriangleup
\hat{Y}(T)-\int\limits_{0}^{T}\hat{\lambda}(t)d(\bigtriangleup\hat{Y}%
(t))-\int\limits_{0}^{T}\bigtriangleup\hat{Y}(t)d\hat{\lambda}(t)\\
-\int\limits_{0}^{T}\bigtriangleup\hat{Z}(t)\tfrac{\partial\hat{H}}{\partial
z}(t)dt-\int\limits_{0}^{T}\int\limits_{\mathbb{R}_{0}}\nabla_{k}\hat
{H}(t,a)\bigtriangleup\hat{K}(t,a)\nu(da)dt].
\end{array}
\label{eq2.7}%
\end{equation}

By $(H_{4})$ all the local martingales involved in $\left(  \ref{eq2.7}%
\right)  $ are martingales up to time $T$, for all $T<\infty$.

Therefore, letting $T\rightarrow\infty$, we obtain by $\left(  \ref{2.5}%
\right)  $%
\begin{equation}%
\begin{array}
[c]{c}%
E[\hat{\lambda}(0)\bigtriangleup\hat{Y}(0)]=\underset{T\rightarrow\infty}%
{\lim}E\mathbf{\ [}\hat{\lambda}(T)\bigtriangleup\hat{Y}(T)]-E[\int
\limits_{0}^{\infty}\{-\hat{\lambda}(t)\bigtriangleup\hat{g}(t)\\
+\bigtriangleup\hat{Y}(t)\tfrac{\partial\hat{H}}{\partial y}(t)+\bigtriangleup
\hat{Z}(t)\tfrac{\partial\hat{H}}{\partial z}(t)+\int\limits_{\mathbb{R}_{0}%
}\nabla_{k}\hat{H}(t,a)\bigtriangleup\hat{K}(t,a)\nu(da)\}dt]\text{.}%
\end{array}
\label{eq2.8}%
\end{equation}

Combining $\left(  \ref{eq2.6}\right)  -\left(  \ref{eq2.8}\right)  $, we
obtain%
\[%
\begin{array}
[c]{l}%
J(\hat{u})-J(u)\geq\underset{T\rightarrow\infty}{\lim}E\mathbf{\ [}%
\hat{\lambda}(T)\bigtriangleup\hat{Y}(T)]+E[\int\limits_{0}^{\infty
}\{\bigtriangleup\hat{H}(t)-\bigtriangleup\hat{b}(t)\hat{p}(t)-\bigtriangleup
\hat{\sigma}(t)\hat{q}(t)\\
-%
{\textstyle\int\limits_{\mathbb{R}_{0}}}
\bigtriangleup\hat{\theta}(t,a)\hat{r}(t,a)\nu(da)-\bigtriangleup\hat
{Y}(t)\tfrac{\partial\hat{H}}{\partial y}(t)-\bigtriangleup\hat{Z}%
(t)\tfrac{\partial\hat{H}}{\partial z}(t)-\int\limits_{\mathbb{R}_{0}}%
\nabla_{k}\hat{H}(t,a)\bigtriangleup\hat{K}(t,a)\nu(da)\}dt]\text{.}%
\end{array}
\]

Since $H$ is concave, we have%
\begin{equation}%
\begin{array}
[c]{l}%
J(\hat{u})-J(u)\geq\underset{T\rightarrow\infty}{\lim}E\mathbf{\ [}%
\hat{\lambda}(T)\bigtriangleup\hat{Y}(T)]+E[\int\limits_{0}^{\infty
}\{\bigtriangleup\hat{X}(t)\frac{\partial\hat{H}}{\partial x}%
(t)+\bigtriangleup\hat{X}_{1}(t)\frac{\partial\hat{H}}{\partial x_{1}}(t)\\
+\bigtriangleup\hat{X}_{2}(t)\frac{\partial\hat{H}}{\partial x_{2}%
}(t)+\bigtriangleup\hat{u}(t)\frac{\partial\hat{H}}{\partial u}%
(t)-\bigtriangleup\hat{b}(t)\hat{p}(t)-\bigtriangleup\hat{\sigma}(t)\hat
{q}(t)-\int\limits_{\mathbb{R}_{0}}\bigtriangleup\hat{\theta}(t,a)\hat
{r}(t,a)\nu(da)\}dt]\text{.}%
\end{array}
\label{eq2.9}%
\end{equation}
Applying now $(H_{1})$, $(H_{4})$ and $(H_{5})$ together with the ItÙ formula
to$%
\begin{array}
[c]{c}%
\mathbf{\ }\hat{p}(t)\bigtriangleup\hat{X}(t),
\end{array}
$ we get%
\begin{equation}%
\begin{array}
[c]{l}%
0\geq\underset{T\rightarrow\infty}{\lim}E\mathbf{\ [\ }\hat{p}%
(T)\bigtriangleup\hat{X}(T)]\\
=E[%
{\textstyle\int\limits_{0}^{\infty}}
\{\bigtriangleup\hat{b}(t)\hat{p}(t)-\bigtriangleup\hat{X}(t)E[\hat{\mu
}(t)\mid\mathcal{F}_{t}]+\bigtriangleup\hat{\sigma}(t)\hat{q}(t)+\int
\limits_{\mathbb{R}_{0}}\bigtriangleup\hat{\theta}(t,a)\hat{r}(t,a)\nu
(da)\}dt]\\
=E[%
{\textstyle\int\limits_{0}^{\infty}}
\{\bigtriangleup\hat{b}(t)\hat{p}(t)-\bigtriangleup\hat{X}(t)\hat{\mu
}(t)+\bigtriangleup\hat{\sigma}(t)\hat{q}(t)+%
{\textstyle\int\limits_{\mathbb{R}_{0}}}
\bigtriangleup\hat{\theta}(t,a)\hat{r}(t,a)\nu(da)\}dt]\text{.}%
\end{array}
\label{eq2.10}%
\end{equation}

By the definition $\left(  \ref{eq1.12}\right)  $ of $\hat{\mu}$ , we have%
\begin{equation}%
\begin{array}
[c]{l}%
E[\mathbf{\ }\int\limits_{0}^{\infty}\bigtriangleup\hat{X}(t)\hat{\mu
}(t)dt]=\underset{T\rightarrow\infty}{\lim}\mathbf{\ }E[\int\limits_{\delta
}^{T+\delta}\bigtriangleup\hat{X}(t-\delta)\hat{\mu}(t-\delta)dt)]\\
=\underset{T\rightarrow\infty}{\lim}E[-%
{\textstyle\int\limits_{\delta}^{T+\delta}}
\frac{\partial\hat{H}}{\partial x}(t-\delta)\bigtriangleup\hat{X}(t-\delta)dt-%
{\textstyle\int\limits_{\delta}^{T+\delta}}
\frac{\partial\hat{H}}{\partial x_{1}}\left(  t\right)  \bigtriangleup\hat
{X}_{1}(t)dt\\
-%
{\textstyle\int\limits_{\delta}^{T+\delta}}
(%
{\textstyle\int\limits_{t-\delta}^{t}}
\frac{\partial\hat{H}}{\partial x_{2}}\left(  s\right)  e^{-\rho s}%
ds)(e^{\rho(t-\delta)}\bigtriangleup\hat{X}(t-\delta))dt].
\end{array}
\label{eq2.11}%
\end{equation}
Using Fubini's theorem and the definition of $X_{2},$ we obtain%
\begin{equation}
\int\limits_{0}^{T}\frac{\partial\hat{H}}{\partial x_{2}}(s)\bigtriangleup
\hat{X}_{2}(s)ds=\int\limits_{\delta}^{T+\delta}(\int\limits_{t-\delta}%
^{t}\frac{\partial\hat{H}}{\partial x_{2}}\left(  s\right)  e^{-\rho
s}ds)e^{\rho(t-\delta)}\bigtriangleup\hat{X}(t-\delta)dt\text{.}
\label{eq2.12}%
\end{equation}

Combining $\left(  \ref{eq2.9}\right)  $\ with $\left(  \ref{eq2.10}\right)
-\left(  \ref{eq2.12}\right)  $, we deduce that%
\[%
\begin{array}
[c]{l}%
J(\hat{u})-J(u)\geq\underset{T\rightarrow\infty}{\lim}E\mathbf{\ [}%
\hat{\lambda}(T)\bigtriangleup\hat{Y}(T)]-\underset{T\rightarrow\infty}{\lim
}E[\hat{p}(T)\bigtriangleup\hat{X}(T)]+E[\int\limits_{0}^{\infty
}\bigtriangleup\hat{u}(t)\frac{\partial\hat{H}}{\partial u}(t)dt]\\
=\underset{T\rightarrow\infty}{\lim}E[\mathbf{\ }\hat{\lambda}%
(T)\bigtriangleup\hat{Y}(T)]-\underset{T\rightarrow\infty}{\lim}%
E\mathbf{\ [}\hat{p}(T)\bigtriangleup\hat{X}(T)]+E[\int\limits_{0}^{\infty
}E\{\bigtriangleup\hat{u}(t)\frac{\partial\hat{H}}{\partial u}(t)\mid
\mathcal{E}_{t}\}dt]\text{.}%
\end{array}
\]

Then%
\[%
\begin{array}
[c]{l}%
J(\hat{u})-J(u)\geq\underset{T\rightarrow\infty}{\lim}\mathbf{\ }%
E[\hat{\lambda}(T)(\hat{Y}(T)-Y(T))]-\underset{T\rightarrow\infty}{\lim
}E\mathbf{\ [}\hat{p}(T)\bigtriangleup\hat{X}(T)]\\
\text{ \ \ \ \ \ \ \ \ \ \ \ \ \ \ \ \ \ \ \ }+E[\int\limits_{0}^{\infty
}E\{\frac{\partial\hat{H}}{\partial u}(t)\mid\mathcal{E}_{t}\}\bigtriangleup
\hat{u}(t)dt]\text{.}%
\end{array}
\]

By assumptions $(H_{1})$ and $(H_{3})$, we conclude $%
\begin{array}
[c]{c}%
J(\hat{u})-J(u)\geq0\text{,}%
\end{array}
$i.e. $\hat{u}$ is an optimal control.
\end{proof}

\section{Necessary conditions of optimality for partial information}

A drawback of the previous section is that the concavity condition is not
always satisfied in applications. In view of this, it is of interest to obtain
conditions for the existence of an optimal control with partial information
where concavity is not needed. We assume the following:

\begin{description}
\item[$(A_{1})$\ ] For all $u\in\mathcal{A}$ and all $\beta\in\mathcal{A}$
bounded, there exists $\mathcal{\epsilon}\mathcal{>}0$ such that
\end{description}

\[
u+s\beta\in\mathcal{A}\text{ \ \ \ \ for all }s\in(-\mathcal{\epsilon
},\mathcal{\epsilon})\text{.}%
\]

This implies in particular that the corresponding solution $X^{u+s\beta}(t)$
of $\left(  \ref{eq1.1}\right)  -\left(  \ref{eq1.5}\right)  $ exists.

\begin{description}
\item[$(A_{2})$] For all $t_{0}>0$, $h>0$ and all bounded $\mathcal{E}_{t_{0}%
}$-measurable random variables $\alpha$, the control process $\beta(t)$
defined by
\end{description}

\begin{equation}
\beta(t)=\alpha1_{\left[  t_{0},t_{0}+h\right)  }(t)\text{ } \label{3.1}%
\end{equation}

belongs to $\mathcal{A}$.

\begin{description}
\item[$(A_{3})$] The following derivative processes exist%
\begin{equation}
\xi(t):=\tfrac{d}{ds}X^{u+s\beta}(t)\mid_{_{s=0}} \label{3.2}%
\end{equation}

\end{description}

\begin{equation}
\phi(t):=\tfrac{d}{ds}Y^{u+s\beta}(t)\mid_{_{s=0}} \label{3.3}%
\end{equation}

\begin{equation}
\eta(t):=\tfrac{d}{ds}Z^{u+s\beta}(t)\mid_{_{s=0}} \label{3.4}%
\end{equation}

\begin{equation}
\psi(t,a):=\tfrac{d}{ds}K^{u+s\beta}(t,a)\mid_{_{s=0}} \label{3.5}%
\end{equation}

\begin{description}
\item[$(A_{4})$ ] We also assume that
\end{description}

\begin{equation}%
\begin{array}
[c]{c}%
E[\int\limits_{0}^{\infty}\{\left\vert \frac{\partial f}{\partial x}%
(t)\xi(t)\right\vert +\left\vert \frac{\partial f}{\partial x_{1}}%
(t)\xi(t-\delta)\right\vert +\left\vert \frac{\partial f}{\partial x_{2}%
}(t)\int\limits_{t-\delta}^{t}e^{-\rho(t-r)}\xi(r)dr\right\vert +\left\vert
\frac{\partial f}{\partial y}(t)\phi(t)\right\vert \\
+\left\vert \frac{\partial f}{\partial z}(t)\eta(t)\right\vert +\left\vert
\frac{\partial f}{\partial u}(t)\beta(t)\right\vert +\int\limits_{\mathbb{R}%
_{0}}\left\vert \nabla_{k}f(t,a)\psi(t,a)\right\vert \nu(da)\}dt]<\infty.
\end{array}
\label{3.6}%
\end{equation}

We can see that%

\[
\tfrac{d}{ds}X_{1}^{u+s\beta}(t)\mid_{_{s=0}}=\xi(t-\delta)
\]

and%

\[
\frac{d}{ds}X_{2}^{u+s\beta}(t)\mid_{_{s=0}}=%
{\textstyle\int\limits_{t-\delta}^{t}}
e^{-\rho(t-r)}\xi(r)dr\text{.\ }%
\]

Note that
\[
\xi(t)=0\text{ for }t\in\left[  -\delta,0\right]  \text{.}%
\]

\begin{theorem}
Assume that $(A_{1})-(A_{4})$ hold. Suppose that $\hat{u}\in\mathcal{A}$ with
corresponding solutions $\boldsymbol{\hat{X}}(t),\hat{Y}(t),\hat{Z}(t),\hat
{K}(t,\cdot),\hat{\lambda}(t),\hat{p}(t),\hat{q}(t)$ and $\hat{r}(t,\cdot)$ of
equations $\left(  \ref{eq1.1}\right)  $, $\left(  \ref{eq1.2}\right)  $,
$\left(  \ref{eq1.10}\right)  $ and $\left(  \ref{eq1.11}\right)  $.

Assume that $\left(  \ref{eq2.4}\right)  $ and the following transversality
conditions hold:$\smallskip\smallskip$

\begin{description}
\item[$(H_{6})$]
\begin{align*}
&  \underset{T\rightarrow\infty}{\lim}E[\mathbf{\ }\hat{p}(T)\xi
(T)]=0\text{,}\\
&  \underset{T\rightarrow\infty}{\lim}E[\hat{\lambda}(T)\phi(T)]=0\text{.}%
\end{align*}

\item[$(H_{7})$] Moreover, assume that the following growth condition holds%
\[%
\begin{array}
[c]{l}%
E[\int\limits_{0}^{T}\{\hat{\lambda}^{2}(t)(\eta^{2}(t)+\int
\limits_{\mathbb{R}_{0}}\psi^{2}(t,a)\nu(da))+\phi^{2}(t)((\frac{\partial
\hat{H}}{\partial z})^{2}(t)+\int\limits_{\mathbb{R}_{0}}\nabla_{k}\hat{H}%
^{2}(t,a)\nu(da))\\
+\hat{p}^{2}(t)(\frac{\partial\sigma}{\partial x}(t)\xi(t)+\frac
{\partial\sigma}{\partial x_{1}}(t)\xi(t-\delta)+\frac{\partial\sigma
}{\partial x_{2}}(t)\int\limits_{t-\delta}^{t}e^{-\rho(t-r)}\xi(r)dr+\frac
{\partial\sigma}{\partial u}(t)\beta(t))^{2}\\
+\hat{p}^{2}(t)(\int\limits_{\mathbb{R}_{0}}\{\frac{\partial\theta}{\partial
x}(t,a)\xi(t)+\frac{\partial\theta}{\partial x_{1}}(t,a)\xi(t-\delta
)+\frac{\partial\theta}{\partial x_{2}}(t,a)\int\limits_{t-\delta}^{t}%
e^{-\rho(t-r)}\xi(r)dr\text{ }\\
+\frac{\partial\theta}{\partial u}(t,a)\beta(t)\}^{2}\nu(da))\}dt]<\infty
,\text{ for all \ }T<\infty\text{.}%
\end{array}
\]

\end{description}

Then the following assertions are equivalent.

\begin{description}
\item[$(i)$ ] For all bounded $\beta\in\mathcal{A}$,%
\[
\frac{d}{ds}J(\hat{u}+s\beta)\mid_{_{s=0}}=0\text{.}%
\]

\item[$(ii)$] For all $t\in\lbrack0,\infty)$,%
\[
E[\frac{\partial}{\partial u}H(t,\boldsymbol{\hat{X}}(t),\hat{Y}(t),\hat
{Z}(t),\hat{K}(t,\cdot),u,\hat{\lambda}(t),\hat{p}(t),\hat{q}(t),\hat
{r}(t,\cdot))\mid\mathcal{E}_{t}{\mathcal{]}}_{u=\hat{u}(t)}=0\text{.}%
\]

\end{description}
\end{theorem}

\begin{proof}
$(i)\Longrightarrow(ii)$:

In the following we use the short- hand notation$%
\begin{array}
[c]{c}%
\frac{\partial b}{\partial x_{i}}(t)=\frac{\partial}{\partial x_{i}%
}b(t,\mathbf{x},u(t))_{\mathbf{x}=\mathbf{X}(t)}\text{ etc; }i=1,2,3.
\end{array}
$

It follows from $\left(  \ref{eq1.1}\right)  $ that%
\[%
\begin{array}
[c]{l}%
d\xi(t)=\{\frac{\partial b}{\partial x}(t)\xi(t)+\frac{\partial b}{\partial
x_{1}}(t)\xi(t-\delta)+\frac{\partial b}{\partial x_{2}}(t)\int
\limits_{t-\delta}^{t}e^{-\rho(t-r)}\xi(r)dr+\frac{\partial b}{\partial
u}(t)\beta(t)\}dt\\
+\{\frac{\partial\sigma}{\partial x}(t)\xi(t)+\frac{\partial\sigma}{\partial
x_{1}}(t)\xi(t-\delta)+\frac{\partial\sigma}{\partial x_{2}}(t)\int
\limits_{t-\delta}^{t}e^{-\rho(t-r)}\xi(r)dr+\frac{\partial\sigma}{\partial
u}(t)\beta(t)\}dB(t)\\
+\int\limits_{\mathbb{R}_{0}}\{\frac{\partial\theta}{\partial x}%
(t,a)\xi(t)+\frac{\partial\theta}{\partial x_{1}}(t,a)\xi(t-\delta
)+\frac{\partial\theta}{\partial x_{2}}(t,a)\int\limits_{t-\delta}^{t}%
e^{-\rho(t-r)}\xi(r)dr+\frac{\partial\theta}{\partial u}(t,a)\beta
(t)\}\tilde{N}(dt,da)\text{,}%
\end{array}
\]

and%
\[%
\begin{array}
[c]{c}%
d\phi(t)=\{-\frac{\partial g}{\partial x}(t)\xi(t)-\frac{\partial g}{\partial
x_{1}}(t)\xi(t-\delta)-\frac{\partial g}{\partial x_{2}}(t)\int
\limits_{t-\delta}^{t}e^{-\rho(t-r)}\xi(r)dr-\frac{\partial g}{\partial
y}(t)\phi(t)\\
-\frac{\partial g}{\partial u}(t)\beta(t)-\frac{\partial g}{\partial z}%
(t)\eta(t)\}dt+\eta(t)dB(t)+\int\limits_{\mathbb{R}_{0}}\psi(t,a)\tilde
{N}(dt,da)\text{,}%
\end{array}
\]

Suppose that assertion $(i)$ holds. Then by $(A_{4})$ and dominated
convergence%
\begin{equation}%
\begin{array}
[c]{l}%
0=\frac{d}{ds}J(\hat{u}+s\beta)\mid_{_{s=0}}\\
=E[\int\limits_{0}^{\infty}\{\frac{\partial f}{\partial x}(t)\xi
(t)+\frac{\partial f}{\partial x_{1}}(t)\xi(t-\delta)+\frac{\partial
f}{\partial x_{2}}(t)\int\limits_{t-\delta}^{t}e^{-\rho(t-r)}\xi
(r)dr+\frac{\partial f}{\partial y}(t)\phi(t)\\
+\frac{\partial f}{\partial z}(t)\eta(t)+\frac{\partial f}{\partial u}%
(t)\beta(t)+\int\limits_{\mathbb{R}_{0}}\nabla_{k}f(t,a)\psi(t,a)\nu
(da)\}dt+h^{^{\prime}}(\hat{Y}(0))\phi(0)]\text{.}%
\end{array}
\label{3.7}%
\end{equation}
$\ \ \ \ $

We know by the definition of $H$ that$\ \ \ \ \ \ \ \ \ $%
\[
\frac{\partial f}{\partial x}(t)=\frac{\partial H}{\partial x}(t)-\frac
{\partial g}{\partial x}(t)\lambda(t)-\frac{\partial b}{\partial
x}(t)p(t)-\frac{\partial\sigma}{\partial x}(t)q(t)-\int\limits_{\mathbb{R}%
_{0}}\frac{\partial\theta}{\partial x}(t,a)r(t,a)\nu(da)
\]
$\ \ $

and similarly for $\frac{\partial f}{\partial x_{1}}(t),\frac{\partial
f}{\partial x_{2}}(t)$, $\frac{\partial f}{\partial u}(t)$, $\frac{\partial
f}{\partial y}(t)$, $\frac{\partial f}{\partial z}(t)$ and $\nabla_{k}f(t,a)$.

By the ItÙ formula and $(H_{7})$, the local martingales which appear after
integration by parts of the process $\hat{\lambda}(t)\phi(t)$ are martingales,
and we get%
\begin{equation}%
\begin{array}
[c]{l}%
E[h^{^{\prime}}(\hat{Y}(0)\phi(0))]=E[\hat{\lambda}(0)\phi(0)]\\
=\underset{T\rightarrow\infty}{\lim}E[\hat{\lambda}(T)\phi(T)]\\
-\underset{T\rightarrow\infty}{\lim}E[\int\limits_{0}^{T}\{\hat{\lambda
}(t)(-\frac{\partial g}{\partial x}(t)\xi(t)-\frac{\partial g}{\partial x_{1}%
}(t)\xi(t-\delta)-\frac{\partial g}{\partial x_{2}}(t)\int\limits_{t-\delta
}^{t}e^{-\rho(t-r)}\xi(r)dr-\frac{\partial g}{\partial y}(t)\phi(t)\\
-\frac{\partial g}{\partial z}(t)\eta(t)-\frac{\partial g}{\partial u}%
(t)\beta(t))+\phi(t)\frac{\partial H}{\partial y}(t)+\eta(t)\frac{\partial
H}{\partial z}(t)+\int\limits_{\mathbb{R}_{0}}\nabla_{k}H(t,a)\psi
(t,a)\nu(da)\}dt]\text{.}%
\end{array}
\label{3.8}%
\end{equation}

Substituting $\left(  \ref{3.8}\right)  $ into $\left(  \ref{3.7}\right)  $ we
get%
\begin{equation}%
\begin{array}
[c]{l}%
0=\frac{d}{ds}J(\hat{u}+s\beta)\mid_{_{s=0}}\\
=E[\int\limits_{0}^{\infty}\{\frac{\partial f}{\partial x}(t)\xi
(t)+\frac{\partial f}{\partial x_{1}}(t)\xi(t-\delta)+\frac{\partial
f}{\partial x_{2}}(t)\int\limits_{t-\delta}^{t}e^{-\rho(t-r)}\xi
(r)dr+\frac{\partial f}{\partial y}(t)\phi(t)\\
+\frac{\partial f}{\partial z}(t)\eta(t)+\frac{\partial f}{\partial u}%
(t)\beta(t)+\int\limits_{\mathbb{R}_{0}}\nabla_{k}f(t,a)\psi(t,a)\nu(da)\\
-\hat{\lambda}(t)(-\frac{\partial g}{\partial x}(t)\xi(t)-\frac{\partial
g}{\partial x_{1}}(t)\xi(t-\delta)-\frac{\partial g}{\partial x_{2}}%
(t)\int\limits_{t-\delta}^{t}e^{-\rho(t-r)}\xi(r)dr-\frac{\partial g}{\partial
y}(t)\phi(t)\\
-\frac{\partial g}{\partial z}(t)\eta(t)-\frac{\partial g}{\partial u}%
(t)\beta(t))+\phi(t)\frac{\partial H}{\partial y}(t)+\eta(t)\frac{\partial
H}{\partial z}(t)+\int\limits_{\mathbb{R}_{0}}\nabla_{k}H(t,a)\psi
(t,a)\nu(da)\}dt]\text{.}%
\end{array}
\label{3.9}%
\end{equation}

Applying the ItÙ formula to the process$%
\begin{array}
[c]{c}%
\hat{p}(t)\xi(t)
\end{array}
$ and using $(H_{7})$, we get%
\begin{equation}%
\begin{array}
[c]{l}%
0=\underset{T\rightarrow\infty}{\lim}E[\hat{p}(T)\mathbf{\ }\xi(T)]\\
=E[\int\limits_{0}^{\infty}\hat{p}(t)\mathbf{\ \{}\frac{\partial b}{\partial
x}(t)\xi(t)+\frac{\partial b}{\partial x_{1}}(t)\xi(t-\delta)+\frac{\partial
b}{\partial x_{2}}(t)\int\limits_{t-\delta}^{t}e^{-\rho(t-r)}\xi
(r)dr+\frac{\partial b}{\partial u}(t)\beta(t)\}dt\\
+\int\limits_{0}^{\infty}\xi(t)E[\mu(t)\mid\mathcal{F}_{t}]dt+\int
\limits_{0}^{\infty}\hat{q}(t)\{\frac{\partial\sigma}{\partial x}%
(t)\xi(t)+\frac{\partial\sigma}{\partial x_{1}}(t)\xi(t-\delta)+\frac
{\partial\sigma}{\partial x_{2}}(t)\int\limits_{t-\delta}^{t}e^{-\rho(t-r)}%
\xi(r)dr+\frac{\partial\sigma}{\partial u}(t)\beta(t)\}dt\\
+\int\limits_{0}^{\infty}\int\limits_{\mathbb{R}_{0}}\hat{r}(t,a)\{\frac
{\partial\theta}{\partial x}(t,a)\xi(t)+\frac{\partial\theta}{\partial x_{1}%
}(t,a)\xi(t-\delta)+\frac{\partial\theta}{\partial x_{2}}(t,a)\int
\limits_{t-\delta}^{t}e^{-\rho(t-r)}\xi(r)dr+\frac{\partial\theta}{\partial
u}(t,a)\beta(t)\}\nu(da)dt]\\
=-\frac{d}{ds}J(\hat{u}+s\beta)\mid_{s=0}+E[\int\limits_{0}^{\infty}%
\frac{\partial H}{\partial u}(t)\beta(t)dt]\text{.}%
\end{array}
\label{3.10}%
\end{equation}

Adding $\left(  \ref{3.9}\right)  $ and $\left(  \ref{3.10}\right)  $ we
obtain%
\[
E[\int\limits_{0}^{\infty}\frac{\partial H}{\partial u}(t)\beta
(t)dt]=0\text{.}%
\]

Now apply this to%
\[
\beta(t)=\alpha1_{\left[  s,s+h\right)  }(t),\text{ }%
\]

where $\alpha$ is bounded and $\mathcal{E}_{t_{0}}$-mesurable, $s\geq t_{0}$.
Then we get%
\[
E[\int\limits_{s}^{s+h}\frac{\partial H}{\partial u}(s)ds\text{ }\alpha]=0.
\]

Differentiating with respect to $h$ at $h=0$ we obtain%
\[
E[\frac{\partial H}{\partial u}(s)\text{ }\alpha]=0.
\]

Since this holds for all $s\geq t_{0}$ and all $\alpha$, we conclude
\[
E[\frac{\partial H}{\partial u}(t_{0})\mid\mathcal{E}_{t_{0}}]=0.
\]
This proves that $(i)$ implies $(ii)$.

$(ii)\Longrightarrow(i)$:

The argument above shows that%
\[
\frac{d}{ds}J(u+s\beta)\mid_{_{s=0}}=E[\int\limits_{0}^{\infty}\frac{\partial
H}{\partial u}(t)\beta(t)dt],
\]

for all $u$, $\beta\in\mathcal{A}$ with $\beta$ bounded. So to complete the
proof we use that every bounded $\beta\in\mathcal{A}$ can be approximated by
linear combinations of controls $\beta$ of the form $\left(  \ref{3.1}\right)
$. We omit the details.
\end{proof}

\section{ Application to optimal consumption with respect to recursive
utility}

\subsection{A general optimal recursive utility problem}

Let $X(t)=X^{(c)}(t)$ be a cash flow modeled by%

\begin{equation}
\left\{
\begin{array}
[c]{l}%
dX(t)=X(t-\delta)[b_{0}(t)dt+\sigma_{0}(t)dB(t)+%
{\displaystyle\int\limits_{\mathbb{R}_{0}}}
\gamma(t,a)\tilde{N}(dt,da)]-c(t)dt;t\geq0,\\
X(0)=x>0,
\end{array}
\right.  \label{4.1}%
\end{equation}

where $b_{0}(t)$, $\sigma_{0}(t)$ and $\gamma(t,a)$ are given bounded
$\mathbb{F}$-predictable processes, $\delta\geq0$ is a fixed delay and
$\gamma(t,a)>-1$ for all $(t,a)\in\left[  0,\infty\right)  \times%
\mathbb{R}
$.

The process $u(t)=c(t)\geq0$ is our control process, interpreted as our
relative consumption rate such that $X^{(c)}(t)>0$ for all $t\geq0$. We let
$\mathcal{A}$ denote the family of all $\mathbb{E}$-predictable relative
consumption rates. To every $c\in\mathcal{A}$ we associate a recursive utility
process $Y^{(c)}(t)=Y(t)$ defined as the solution of the infinite horizon BSDE%

\begin{equation}
Y(t)=E[Y(T)+%
{\textstyle\int\limits_{t}^{T}}
g\left(  s,Y(s),c(s)\right)  ds\mid\mathcal{F}_{t}]\text{ for all }t\leq
T\text{,} \label{4.2}%
\end{equation}

valid for all deterministic $T<\infty$. The number $Y^{(c)}(0)$ is called the
recursive utility of consumption process $c(t)$; $t\geq0$ (See e.g. Duffie \&
Epstein $(1992),$ \cite{DE}).

Suppose the solution $(Y,Z,K)$ of the infinite horizon BSDE $\left(
\ref{4.2}\right)  $ satisfies the condition $\left(  \ref{eq1.4}\right)  $ and
let $c(s);s\geq0$ be the consumption rate.

We assume that the function $%
\begin{array}
[c]{c}%
g(t,y,c):%
\mathbb{R}
_{+}^{3}\rightarrow%
\mathbb{R}%
\end{array}
$ satisfies the following conditions:

\begin{enumerate}
\item $g(t,y,c)$ is concave with respect to $y$ and $c$

\item
\begin{equation}%
{\textstyle\int\limits_{0}^{T}}
E\left[  \left\vert g(s,Y(s),c(s))\right\vert \right]  ds<\infty,\text{ for
all }c\in\mathcal{A}\text{, }T<\infty. \label{4.3}%
\end{equation}

\item $\frac{\partial}{\partial c}g(t,y,c)$ has an inverse:
\end{enumerate}

\[
I(t,v,y)=\left\{
\begin{array}
[c]{l}%
0\text{ \ \ \ \ \ \ \ \ \ \ \ \ \ \ \ \ \ \ \ \ \ \ if }v\geq v_{0}(t,y),\\
(\dfrac{\partial}{\partial c}g(t,y,c))^{-1}(v)\text{ if }0\leq v\leq
v_{0}(t,y),
\end{array}
\right.
\]

where $%
\begin{array}
[c]{c}%
v_{0}(t,y)=\frac{\partial}{\partial c}g(t,y,0)
\end{array}
$.

We want to maximize the recursive utility $Y^{(c)}(0)$, i.e. we want to find
$c^{\ast}\in\mathcal{A}$ such that%

\begin{equation}
\underset{c\in\mathcal{A}}{\sup}Y^{(c)}(0)=Y^{(c^{\ast})}(0)\text{.}
\label{4.4}%
\end{equation}

We call such a process $c^{\ast}$ an optimal recursive utility consumption
rate. \newline\newline We see that the problem $\left(  \ref{4.4}\right)  $ is
a special case of problem $\left(  \ref{eq1.8}\right)  $ with%

\[
J(u)=Y(0),
\]

$%
\begin{array}
[c]{c}%
f=0
\end{array}
$, $%
\begin{array}
[c]{c}%
h(y)=y
\end{array}
$, $%
\begin{array}
[c]{c}%
u=c
\end{array}
$ and%

\[%
\begin{array}
[c]{l}%
b(t,\boldsymbol{x},c)=x_{1}b_{0}(t)-c,\\
\sigma(t,\boldsymbol{x},u)=x_{1}\sigma_{0}(t),\\
\theta(t,\boldsymbol{x},u,a)=x_{1}\gamma(t,a).
\end{array}
\]

In this case the Hamiltonian defined in $\left(  \ref{eq1.9}\right)  $ takes
the form%

\begin{align}
H(t,\boldsymbol{x},y,z,c,\lambda,p,q,r(\cdot))  &  =\lambda g(t,y,c)+\left(
x_{1}b_{0}(t)-c\right)  p\nonumber\\
&  +x_{1}\sigma_{0}(t)q+x_{1}\int\limits_{\mathbb{R}_{0}}\gamma(t,a)r(a)\nu
(da). \label{4.5}%
\end{align}

Maximizing $E[H\mid\mathcal{E}_{t}]$ as a function of $c$\ gives the first
order condition%

\begin{equation}
E[\lambda(t)\frac{\partial g}{\partial c}(t,Y(t),c(t))\mid\mathcal{E}%
_{t}]=E[p(t)\mid\mathcal{E}_{t}], \label{4.6}%
\end{equation}

for an optimal $c(t)$.

The pair of adjoint processes $\left(  \ref{eq1.10}\right)  -\left(
\ref{eq1.11}\right)  $ is given by%

\begin{equation}
\left\{
\begin{array}
[c]{l}%
d\lambda(t)=\lambda(t)\dfrac{\partial g}{\partial y}(t,Y(t),c(t))dt,\\
\lambda(0)=1,
\end{array}
\right.  \label{4.7}%
\end{equation}

and%

\begin{equation}
dp(t)=E[\mu(t)\mid\mathcal{F}_{t}]dt+q(t)dB(t)+\int\limits_{\mathbb{R}_{0}%
}r(t,a)\tilde{N}(dt,da);t\in\left[  0,\infty\right)  , \label{4.8}%
\end{equation}

where%
\begin{equation}
\mu(t)=-[b_{0}(t+\delta)p(t+\delta)+\sigma_{0}(t+\delta)q(t+\delta
)+\int\limits_{\mathbb{R}_{0}}\gamma(t+\delta,a)r(t+\delta,a)\nu(da)].
\label{4.9}%
\end{equation}

with terminal condition as in $\left(  \ref{eq1.4}\right)  $, i.e.%

\begin{equation}
E\underset{t\geq0}{[\sup\text{ }}e^{\kappa t}p^{2}(t)+%
{\textstyle\int\limits_{0}^{\infty}}
e^{\kappa s}(q^{2}(s)+%
{\textstyle\int\limits_{\mathbb{R}_{0}}}
r^{2}\left(  s,a\right)  \nu(da))ds]<\infty,\nonumber
\end{equation}
for all constants $\kappa>0$.

Equation $\left(  \ref{4.7}\right)  $ has the solution%
\begin{equation}
\lambda(t)=\exp(%
{\textstyle\int\limits_{0}^{t}}
\frac{\partial g}{\partial y}(s,Y(s),c(s))ds);t\geq0 \label{4.10}%
\end{equation}

which substituted into $\left(  \ref{4.6}\right)  $ gives%

\begin{equation}
E[\tfrac{\partial g}{\partial c}(t,Y(t),c(t))\exp(%
{\textstyle\int\limits_{0}^{t}}
\tfrac{\partial g}{\partial y}(s,Y(s),c(s))ds)\mid\mathcal{E}_{t}%
]=E[p(t)\mid\mathcal{E}_{t}]. \label{4.11}%
\end{equation}

We refer to Theorem $5.1$ in \cite{AHOP} for a proof of the existence of the
solution of the ABSDE $\left(  \ref{4.8}\right)  $.

\subsection{A solvable special case}

In order to get a solvable case we choose the driver $g$ in $\left(
\ref{4.2}\right)  $ to be of the form%
\begin{equation}
g(t,y,c)=-\alpha(t)y+\ln c, \label{4.12}%
\end{equation}

where $\alpha(t)\geq\alpha>0$ is an $\mathbb{F}$-adapted process.

We also choose%
\begin{equation}
\delta=0\text{ and }\mathcal{E}_{t}=\mathcal{F}_{t};t\geq0, \label{4.13}%
\end{equation}

and we represent the consumption rate $c(t)$ as%
\begin{equation}
c(t)=\rho(t)X(t)\text{,} \label{4.14}%
\end{equation}

where $%
\begin{array}
[c]{c}%
\rho(t)=\frac{c(t)}{X(t)}\geq0
\end{array}
$ is the relative consumption rate.

We restrict our attention to processes $c$ such that the wealth process,
solution of $\left(  \ref{4.1}\right)  $, is strictly positive and $\rho$ is
bounded away from $0$. This set of controls $\rho$ is denoted by $\mathcal{A}$.

The FBSDE system now has the form%
\begin{equation}
\left\{
\begin{array}
[c]{l}%
dX(t)=X(t^{-})[(b_{0}(t)-\rho(t))dt+\sigma_{0}(t)dB(t)+%
{\displaystyle\int\limits_{\mathbb{R}_{0}}}
\gamma(t,a)\tilde{N}(dt,da)];t\geq0,\\
X(0)=x>0,
\end{array}
\right.  \label{4.15}%
\end{equation}

and
\begin{equation}
Y(t)=Y^{(\rho)}(t)=E[Y(T)+%
{\textstyle\int\limits_{t}^{T}}
\left(  -\alpha(s)Y(s)+\ln\rho(s)X(s)\right)  ds\mid\mathcal{F}_{t}],
\label{4.16}%
\end{equation}

i.e.%
\begin{equation}
dY(t)=-\left(  -\alpha(t)Y(t)+\ln\rho(t)+\ln(X(t))\right)  dt+Z(t)dB(t);t\geq
0. \label{4.17}%
\end{equation}

We want to find $\rho^{\ast}\in\mathcal{A}$ such that%
\begin{equation}
\underset{\rho\in\mathcal{A}}{\sup}Y^{(\rho)}(0)=Y^{(\rho^{\ast})}(0).
\label{4.18}%
\end{equation}

In this case the Hamiltonian $\left(  \ref{eq1.9}\right)  $ gets the form%
\begin{align}
H(t,x,y,\rho,\lambda,p,q,r)  &  =\lambda(-\alpha(t)y+\ln\rho+\ln x)+x\left(
b_{0}(t)-\rho\right)  p\nonumber\\
&  +x\sigma_{0}(t)q+x\int\limits_{\mathbb{R}_{0}}\gamma(t,a)r(a)\nu(da).
\label{4.19}%
\end{align}

Maximizing $H$ with respect to $\rho$ gives the first order equation%
\begin{equation}
\lambda(t)\frac{1}{\rho(t)}=p(t)X(t), \label{4.20}%
\end{equation}

\bigskip

where, by $(1.10)-(1.11)$ $\lambda(t)$ and $(p(t),q(t),r(t,\cdot))$ satisfy
the FBSDEs%
\begin{equation}
\left\{
\begin{array}
[c]{l}%
d\lambda(t)=-\alpha(t)\lambda(t)dt,\\
\lambda(0)=1,
\end{array}
\right.  \label{4.21}%
\end{equation}

and%

\begin{equation}%
\begin{array}
[c]{l}%
dp(t)=-[\lambda(t)\frac{1}{X(t)}+\left(  b_{0}(t)-\rho(t)\right)
p(t)+\sigma_{0}(t)q(t)\\
+\int\limits_{\mathbb{R}_{0}}\gamma(t,a)r(a)\nu(da)]dt+q(t)dB(t)+\int
\limits_{\mathbb{R}_{0}}r(t,a)\tilde{N}(dt,da),
\end{array}
\label{4.22}%
\end{equation}

with terminal condition as in $\left(  \ref{eq1.4}\right)  $, i.e.%

\begin{equation}
E\underset{t\geq0}{[\sup\text{ }}e^{\kappa t}p^{2}(t)+%
{\textstyle\int\limits_{0}^{\infty}}
e^{\kappa s}(q^{2}(s)+%
{\textstyle\int\limits_{\mathbb{R}_{0}}}
r^{2}\left(  s,a\right)  \nu(da))ds]<\infty, \label{4.23}%
\end{equation}

for all constants $\kappa>0$.

The infinite horizon BSDE $\left(  \ref{4.22}\right)  -\left(  \ref{4.23}%
\right)  $ has a unique solution, (see e.g. Theorem $3.1$ in \cite{HOP}).

Then, the solutions of $\left(  \ref{4.21}\right)  -\left(  \ref{4.22}\right)
$ are respectively,%
\begin{equation}
\lambda(t)=\exp(-%
{\textstyle\int\limits_{0}^{t}}
\alpha(s)ds), \label{4.24}%
\end{equation}

and, for all $0\leq t\leq T$ and all $T<\infty$,%
\begin{equation}
p(t)\Gamma(t)=E[p(T)\Gamma(T)+%
{\textstyle\int\limits_{t}^{T}}
\lambda(s)\frac{\Gamma(s)}{X(s)}ds\mid\mathcal{F}_{t}], \label{4.25}%
\end{equation}

where $\Gamma(t)$ is given by%
\begin{equation}
\left\{
\begin{array}
[c]{l}%
d\Gamma(t)=\Gamma(t^{-})[\left(  b_{0}(t)-\rho(t)\right)  dt+\sigma
_{0}(t)dB(t)+%
{\displaystyle\int\limits_{\mathbb{R}_{0}}}
\gamma(t,a)\tilde{N}(dt,da)];t\geq0,\\
\text{\ }\Gamma(0)=1.
\end{array}
\right.  \label{4.26}%
\end{equation}

(See e.g. \cite{OSA}).

This gives%
\begin{equation}%
\begin{array}
[c]{c}%
\Gamma(t)=\exp(-%
{\textstyle\int\limits_{0}^{t}}
\sigma_{0}(s)dB(s)+%
{\textstyle\int\limits_{0}^{t}}
\{b_{0}(s)-\rho(s)-\frac{1}{2}\sigma_{0}^{2}(s)\}ds\\
+%
{\textstyle\int\limits_{0}^{t}}
\int\limits_{\mathbb{R}_{0}}\{\ln(1+\gamma(s,a))-\gamma(s,a)\}\nu(da)ds\\
+%
{\textstyle\int\limits_{0}^{t}}
\int\limits_{\mathbb{R}_{0}}\ln(1+\gamma(s,a))\tilde{N}(ds,da);t\geq0.
\end{array}
\label{4.27}%
\end{equation}

Comparing with $\left(  \ref{4.15}\right)  $ we see that%
\begin{equation}
X(t)=x\Gamma(t);t\geq0. \label{4.28}%
\end{equation}

Substituting this into $\left(  \ref{4.25}\right)  $ we obtain%
\begin{equation}
p(t)X(t)=E[p(T)X(T)+%
{\textstyle\int\limits_{t}^{T}}
\exp(-%
{\textstyle\int\limits_{0}^{s}}
\alpha(r)dr)ds\mid\mathcal{F}_{t}]. \label{4.29}%
\end{equation}

Since $\rho$ is bounded away from $0$ we deduce from $\left(  \ref{4.20}%
\right)  $ that%
\begin{equation}
p(T)X(T)=\frac{\lambda(T)}{\rho(T)}=\frac{1}{\rho(T)}\exp(-%
{\textstyle\int\limits_{0}^{T}}
\alpha(r)dr)\rightarrow0\text{ dominatedly as }T\rightarrow\infty.
\label{4.30}%
\end{equation}

Hence, by letting $T\rightarrow\infty$ in $\left(  \ref{4.29}\right)  $ we
get
\begin{equation}
p(t)X(t)=E[%
{\textstyle\int\limits_{t}^{\infty}}
\exp(-%
{\textstyle\int\limits_{0}^{s}}
\alpha(r)dr)ds\mid\mathcal{F}_{t}]. \label{4.31}%
\end{equation}

This implies that $p(t)>0$ and hence $\rho(t)$ given by $\left(
\ref{4.20}\right)  $ is indeed a maximum point of $H$.

By $(4.20)$ we therefore get the following candidate for the optimal relative
consumption rate%
\begin{equation}
\rho(t)=\rho^{\ast}(t)=\frac{\exp(-%
{\textstyle\int\limits_{0}^{t}}
\alpha(r)dr)}{E[%
{\textstyle\int\limits_{t}^{\infty}}
\exp(-%
{\textstyle\int\limits_{0}^{s}}
\alpha(r)dr)ds\mid\mathcal{F}_{t}]};t\geq0, \label{4.32}%
\end{equation}

If $\alpha$ is such that this expression for $\rho^{\ast}(t)$ is bounded away
from $0$, then $\rho^{\ast}$ is optimal. Note that the corresponding optimal
net cash flow $X^{\ast}(t)$ is given by%
\begin{equation}
X^{\ast}(t)=x\exp(%
{\textstyle\int\limits_{0}^{t}}
\sigma_{0}(s)dB(s)+%
{\textstyle\int\limits_{0}^{t}}
\{b_{0}(s)-\rho(s)-\frac{1}{2}\sigma_{0}^{2}(s)\}ds);t\geq0. \label{4.33}%
\end{equation}

In particular, $X^{\ast}(t)>0$ for all $t\geq0$, as required.

In particular, if $\alpha(r)=\alpha>0$ (constant) for all $r$, then%
\begin{equation}
\rho^{\ast}(t)=\alpha;t\geq0\text{.} \label{4.34}%
\end{equation}

With this choice of $\rho^{\ast}$ we see by $\left(  \ref{4.31}\right)  $,
$\left(  \ref{4.24}\right)  $ and condition $\left(  \ref{eq1.4}\right)  $ for
$Y(t)$ that the transversality conditions $(H_{3})$ and $(H_{6})$ hold, and we
have proved:

\begin{theorem}
The optimal relative consumption rate $\rho^{\ast}$ $(t)$ for problem $\left(
\ref{4.12}\right)  -\left(  \ref{4.18}\right)  $ is given by $\left(
\ref{4.32}\right)  $, provided that $\rho^{\ast}$ $(t)$ is bounded away from
$0$.

In particular, if $\alpha(r)=\alpha>0$ (constant) for all $r$, then
$\rho^{\ast}(t)=\alpha$ for all $t$.
\end{theorem}

\begin{description}
\item[Acknowledgment] We want to thank Brahim Mezerdi for helpful discussions.
\end{description}

\end{document}